\documentclass[12pt,reqno]{article}
\usepackage[letterpaper,margin=1.6cm]{geometry}
\usepackage{latexsym, amsfonts, amsthm,amssymb,amscd,amsmath,makeidx}
\usepackage{enumerate}
\usepackage[normalem]{ulem}
\usepackage{exscale}
\usepackage{overpic} 
\usepackage{color} 
\usepackage{graphicx}
\usepackage{overpic}  
\usepackage{bm}
\usepackage{comment}
\usepackage{caption}
\usepackage{subcaption}
\usepackage{float}
\usepackage{mathrsfs}
\usepackage{amssymb}
\usepackage{mathtools}

\usepackage{tikz}
\usepackage{blindtext}
\usepackage{bbm}
\usepackage{fancyhdr} 
\usepackage{mathrsfs}
\usepackage{wrapfig}
\usepackage{hyperref}
\usepackage{cleveref}
\usepackage{cases}

\hypersetup{
    colorlinks,
    citecolor=black,
    filecolor=black,
    linkcolor=blue,
    urlcolor=black
}

\definecolor{DarkGreen}{rgb}{0.2,0.6,0.2}

\def\ua{\uparrow}
\def\da{\downarrow}

\def\ignore#1{}

\def\bR{{\mathbb R}}

\def\bN{\mathbb N}

\def\bP{{\mathbb P}}
\def\bE{{\mathbb E}}

\numberwithin{equation}{section}

\usetikzlibrary{positioning,calc,cd}

\allowdisplaybreaks[4]

\def\var{\mbox{var}}

\newtheorem{theorem}{Theorem}[section]
\newtheorem{proposition}[theorem]{Proposition}
\newtheorem{lemma}[theorem]{Lemma}
\newtheorem{corollary}[theorem]{Corollary}

\theoremstyle{definition}

\def\<{\langle}
\def\>{\rangle}

 \def\ua{\uparrow}
 \def\da{\downarrow}

\def\var{\text{\rm Var}}

\begin{document}

\title{A Gladyshev theorem for trifractional Brownian motion and $n$-th order fractional Brownian motion}
\author{ Xiyue Han\thanks{Department of Statistics and Actuarial Science, University of Waterloo. E-mail: {\tt xiyue.han@uwaterloo.ca}} }
        
\date{\normalsize May 04, 2021}

\vspace{-1cm}

\maketitle

\begin{abstract}We prove limit theorems for the weighted quadratic variation of trifractional Brownian motion and $n$-th order fractional Brownian motion. Furthermore, a sufficient condition for the $L^P$-convergence of the weighted quadratic variation for Gaussian processes is obtained as a byproduct. As an application, we give a statistical estimator for the self-similarity index of trifractional Brownian motion. These theorems extend results of Baxter, Gladyshev, and Norvai\v{s}a. \end{abstract}


\medskip


\section{Introduction}
The study of limit theorems for the weighted quadratic variation of Gaussian processes was initiated by Baxter \cite{baxter1956strong}. Later, Gladyshev \cite{gladyshev1961new} established a similar result under more general assumptions, which applied to a wide class of Gaussian processes, for example, the fractional Brownian motion. This general result is often referred to as Gladyshev's theorem, which also provides a method for estimating the self-similarity index for the fractional Brownian motion. 

Consequently, Gladyshev's theorem has been extended in various directions by many authors, see Klein and Gin\'{e} \cite{klein1975quadratic}, Marcus and Rosen \cite{marcus1992p}, and K\^{o}no \cite{kono1969oscillation}. A recent preeminent complement to Gladyshev's theorem was proposed by Norvai\v{s}a \cite{norvaivsa2011complement}, who replaced hypotheses in \cite{gladyshev1961new} with weaker ones. The most recent study on this topic was due to Viitasaari \cite{viitasaari2019necessary} who gave a necessary and sufficient condition for the limit theorem for the quadratic variation, which greatly simplified the existing methodology. These results apply to a large number of Gaussian processes. For instance, both \cite{norvaivsa2011complement} and \cite{viitasaari2019necessary} directly yield the limit theorem for the weighted quadratic variation of the bifractional Brownian motion. However, existing techniques do not apply to the trifractional Brownian motion (tri-fBm) and the $n$-th order fractional Brownian motion ($n$-fBm); see the discussion in Section \ref{Section_Main}.

In a recent paper, Lei and Nualart \cite{lei2009decomposition} showed that a bifractional Brownian motion could be decomposed into the sum of a fractional Brownian motion and a Gaussian process with absolutely continuous trajectories. This process was subsequently rediscovered and studied by Ma \cite{ma2013schoenberg}, who also referred to it as the trifractional Brownian motion. According to Ma \cite{ma2013schoenberg}, a \emph{trifractional Brownian motion} $Z_{H,K} := \{Z_{H,K}(t),t\in [0,\infty)\}$ with parameters $H \in (0,1)$ and $K \in (0,1)$ is a centered Gaussian process with covariance function $C_{H,K}$,
\begin{equation*}
	C_{H,K}(s,t) = t^{2HK} + s^{2HK} - (t^{2H}+s^{2H})^{K},
\end{equation*}
for $t,s \in [0,\infty)$. A trifractional Brownian motion is an $HK$-self-similar process with non-stationary increments. 

The well-known fractional Brownian motion is governed by the Hurst parameter $H$ between $0$ and $1$. However, many observations reveal that the Hurst parameter could be larger than one in real life. Motivated by this fact, Perrin et al. \cite{perrin2001nth} introduced the $n$-th order fractional Brownian motion as an extension to the fractional Brownian motion. An \emph{$n$-th order fractional Brownian motion} $B_{H,n} := \{B_{H,n}(t),t \in [0,\infty)\}$ is a centered Gaussian process with the following covariance function, 
\begin{equation*}
	G_{H,n}(s,t) = (-1)^n\frac{C_{H}^n}{2}\Bigg(|t-s|^{2H} - \sum_{j = 0}^{n-1}(-1)^j{2H \choose j}\bigg(\big(\frac{t}{s}\big)^js^{2H} + \big(\frac{s}{t}\big)^jt^{2H}\bigg)\Bigg),
\end{equation*}
where $n \in \bN$, $H \in (n-1,n)$, and $C_H^{n} = \big(\Gamma(2H+1)|\sin(\pi H)|\big)^{-1}$. The class of $n$-th order fractional Brownian motions allows a wider range of the Hurst parameter $H$. In other words, the class of $n$-th order fractional Brownian motions extends the Hurst parameter beyond the constraint $H \in (0,1)$ and includes the case of fractional Brownian motion for $n=1$. Besides, an $n$-th order fractional Brownian motion is still $H$-self-similar and $n$-stationary.

For a real-valued process $X$ and $\alpha \in \bR$, we study the limit of 
\begin{equation}\label{eq_weight_sum}
	2^{\alpha n}\sum_{k = 1}^{2^n}\bigg(X\Big(\frac{k}{2^n}\Big) - X\Big(\frac{k-1}{2^n}\Big)\bigg)^2 
\end{equation}
as $n \ua \infty$, provided it exists in some sense. In particular, our study concerns the convergence results of \eqref{eq_weight_sum} for the trifractional Brownian motion and $n$-th order fractional Brownian motion for various $\alpha$. Our results also lead to a statistical estimator for the self-similarity index for the trifractional Brownian motion.
\section{Statement of main results}\label{Section_Main}
For a real-valued stochastic process $X= \{X(t), t \in [0,\infty)\}$, if there exists $\gamma \in \bR$ such that 
\begin{equation}\label{sharp_eq}
	\lim\limits_{n \ua \infty } 2^{\alpha n} \sum_{k = 1}^{2^n}\bigg(X\Big(\frac{k}{2^n}\Big) - X\Big(\frac{k-1}{2^n}\Big)\bigg)^2 = \begin{cases} 0 &\quad\text{for} \quad\alpha < \gamma,  \\ \infty &\quad\text{for} \quad\alpha > \gamma,\end{cases} 
\end{equation}
almost surely, we then say $\gamma$ is the \emph{critical exponent} of the weighted quadratic variation of the process $X$. This notion of the critical exponent measures the roughness of the process $X$; the rougher a process is, the smaller its critical exponent will be. At the critical case $\alpha = \gamma$, the left-hand side of \eqref{sharp_eq} may be infinite, finite or nonexistent. We refer to \cite{HanSchiedZhang1, HanSchiedZhang2} for a discussion of what can happen for deterministic fractal functions.

For Gaussian processes, their critical exponents can be computed by Gladyshev's theorem and its succeeding extensions. For instance, in the very recent complement \cite{norvaivsa2011complement}, Norvai\v{s}a considers Gaussian processes $X$ with structure functions $\psi_X$ of the following form,
\begin{equation*}
	\psi_X(s,t):= \bE\big((X(t)-X(s))^2\big) = d|t-s|^{2-\lambda} + b(s,t) \qquad \text{for} \qquad t,s \in [0,\infty),
\end{equation*}
where $d > 0$, $\lambda \in (0,2)$, and $b(\cdot,\cdot)$ is a symmetric function such that for all $\epsilon > 0$
\begin{equation*}
	\lim _{h \downarrow 0} \frac{\sup \{|b(t-h, t)|: t \in(\epsilon, 1]\}}{h^{2-\lambda}}=0.
\end{equation*}
Under these conditions, Norvai\v{s}a shows that the Gaussian process $X$ admits the critical exponent $\gamma = 1 -\lambda$, and the weighted quadratic variation converges to $d$ almost surely at the critical case $\alpha = \gamma$. For instance, the results in \cite{norvaivsa2011complement} directly apply to the bifractional Brownian motion, which has a critical exponent $2HK-1$.

However, as stated in previous paragraphs, these existing methods fail to calculate the critical exponent for the trifractional Brownian motion and the $n$-th order fractional Brownian motion. For instance, the trifractional Brownian motion $Z_{H,K}$ has the following structure function,
\begin{equation*}
	\psi_{Z_{H,K}}(s,t) = 2(s^{2H}+t^{2H})^K - 2^{K}s^{2HK} -2^Kt^{2HK},
\end{equation*}
for $t,s \in [0,\infty)$. It is clear that $d = 0$ in this case. Moreover, for the $n$-th fractional Brownian motion $B_{H,n}$, one has \begin{equation}\label{n_fbm_eq_1}
	\psi_{B_{H,n}}(s,t) = (-1)^nC_H^n\Bigg(\sum_{j = 0}^{n-1}(-1)^j{2H \choose j}\bigg(t^js^{2H-j} + s^jt^{2H-j} -t^{2H} - s^{2H}\bigg) - |t-s|^{2H}\Bigg),
\end{equation}
for $t,s \in [0,\infty)$. The constant $d = (-1)^{n+1}C^n_H$ could possibly take negative values. 

The rest of this paper is organized as follows. We establish \eqref{sharp_eq} and compute the critical exponent for the trifractional Brownian motion and the $n$-th order fractional Brownian motion in Theorem \ref{Thm_tri_new} and Theorem \ref{nfbm_thm}, respectively. Proposition \ref{Lemma_Lp} studies the $L^P$-convergence of the weighted quadratic variation, which is needed in the proofs of subsequent theorems. We then prove the limiting theorems at the critical cases for each process in Theorem \ref{thm_tri_ip} and Theorem \ref{Thm_nfbm_ip}. Finally, Corollary \ref{cor_estimate} constructs a consistent estimator for the self-similarity index of the trifractional Brownian motion. Proofs are given in Section \ref{Section_Proofs}. 

Now we state our main results, let us start with the trifractional Brownian motion.

\begin{theorem}\label{Thm_tri_new}
	For $H \in (0,1)$, $K \in (0,1)$ and $T \in (0,\infty)$, the trifractional Brownian motion $Z_{H,K}$ has the following limiting behavior, 
	\begin{enumerate}
		\item \label{Tri_part_a} If $HK < \frac{1}{2}$,
		\begin{numcases}{ \lim_{n \ua \infty}2^{\alpha n} \sum_{k = 1}^{2^n}\bigg(Z_{H,K}\Big(\frac{kT}{2^n}\Big) - Z_{H,K}\Big(\frac{(k-1)T}{2^n}\Big)\bigg)^2 = }
		0 \quad \text{ for}& $\alpha < 2HK,$ \label{tri_case_eq_1_1}\\
		\infty \quad \text{for}& $\alpha > 2HK,$ \label{tri_case_eq_1_2}
		\end{numcases}
		almost surely. In particular, the critical exponent $\gamma = 2HK$.
		\item \label{Tri_part_b} If $HK \ge \frac{1}{2}$,
		\begin{numcases}{\lim\limits_{n \ua \infty } 2^{\alpha n}\sum_{k = 1}^{2^n}\bigg(Z_{H,K}\Big(\frac{kT}{2^n}\Big) - Z_{H,K}\Big(\frac{(k-1)T}{2^n}\Big)\bigg)^2 = }
			0 \quad \text{ for} & $\alpha < 1,$ \label{tri_case_eq_2_1}\\
			\infty \quad \text{for} &$\alpha > 1,$ \label{tri_case_eq_2_2}
		\end{numcases}
		almost surely. In particular, the critical exponent $\gamma = 1$.
	\end{enumerate}
\end{theorem}
Now we state the result for the $n$-th order fractional Brownian motion. 
\begin{theorem}\label{nfbm_thm}
	For $m \ge 2$, $H \in (m-1,m)$ and $T \in (0, \infty)$, the $m$-th order fractional Brownian motion $B_{H.m}$ has the following limiting behavior, 
	\begin{equation}\label{nthm_thm_eq_1}
		\lim_{n \ua \infty}2^{\alpha n} \sum_{k = 1}^{2^n}\bigg(B_{H.m}\Big(\frac{kT}{2^n}\Big) - B_{H.m}\Big(\frac{(k-1)T}{2^n}\Big)\bigg)^2 = 0 \quad \text{for} \quad \alpha < 1,
	\end{equation}
	almost surely. Moreover, if $m \ge 3$, we have 
	\begin{equation}\label{nthm_thm_eq_2}
		\lim_{n \ua \infty}2^{\alpha n} \sum_{k = 1}^{2^n}\bigg(B_{H.m}\Big(\frac{kT}{2^n}\Big) - B_{H.m}\Big(\frac{(k-1)T}{2^n}\Big)\bigg)^2 = \infty \quad \text{for} \quad \alpha > 1,
	\end{equation}
	almost surely. In particular, the critical exponent $\gamma = 1$ for $m \ge 3$. \end{theorem}
The following proposition is needed to study limit theorems at the critical cases, i.e., $\alpha = 1$ in part \ref{Tri_part_b} of Theorem \ref{Thm_tri_new} and $\alpha = 1$ in Theorem \ref{nfbm_thm}. In Proposition \ref{Lemma_Lp}, we provide a sufficient condition for the $L^P$-convergence of \eqref{eq_weight_sum} for $\alpha = 1$ as $n \ua \infty$.
\begin{proposition}\label{Lemma_Lp}
	Let $X$ be a centered Gaussian process. If 
	\begin{equation}\label{lemma_gen_bound}
		\sup_{n\in \bN}2^{n}\sum_{k = 1}^{2^n}\bE\left[\bigg(X\Big(\frac{k}{2^n}\Big)-X\Big(\frac{k-1}{2^n}\Big)\bigg)^2\right] < \infty,
	\end{equation}
	the sum 
	\begin{equation*}
		2^{n}\sum_{k = 1}^{2^n}\left(X\Big(\frac{k}{2^n}\Big)-X\Big(\frac{k-1}{2^n}\Big)\right)^2
	\end{equation*}
	converges to some non-constant random variable in $L^p$ as $n \ua \infty$ for all $p \in (1,\infty)$.
\end{proposition}
For the convenience of exposition in subsequent proofs, we introduce the following notation: For a centered Gaussian process $X$, we denote 
\begin{equation}\label{eq_phi_def}
	\phi^{(m,n)}_{j,k}:= \bE\left[\Big(X\Big(\frac{j}{2^m}\Big)-X\Big(\frac{j-1}{2^m}\Big)\Big)\Big(X\Big(\frac{k}{2^n}\Big)-X\Big(\frac{k-1}{2^n}\Big)\Big)\right],
\end{equation}
for $m,n \in \bN$, $1 \le j \le 2^m$ and $1 \le k \le 2^n$. In particular, when the arguments $m = n$, we denote $\phi^{(n)}_{j,k}:= \phi^{(n,n)}_{j,k}$ in short. In Section \ref{Section_Proofs}, we assign the Gaussian process $X$ to be either the trifractional Brownian motion or the $n$-th order fractional Brownian motion. 
\begin{proof}[Proof of Proposition \ref{Lemma_Lp}]
	According to \cite[Theorem 2.1]{viitasaari2019necessary}, it suffices to show that the double limit $$\lim_{n,m \ua \infty}2^{n+m}\sum_{k = 1}^{2^n}\sum_{j = 1}^{2^m}(\phi^{(m,n)}_{j,k})^2$$
	exists. For simplicity, let $a_{m,n}:= 2^{n+m}\sum_{k = 1}^{2^n}\sum_{j = 1}^{2^m}(\phi^{(m,n)}_{j,k})^2$. The Cauchy-Schwarz inequality leads to 
	\begin{equation*}
		(\phi_{j,k}^{(m,n)})^2 \le \bE\left[\Big(X\Big(\frac{k}{2^n}\Big)-X\Big(\frac{k-1}{2^n}\Big)\Big)^2\right]\bE\left[\Big(X\Big(\frac{j}{2^m}\Big)-X\Big(\frac{j-1}{2^m}\Big)\Big)^2\right]= \phi^{(m)}_{j,j}\phi^{(n)}_{k,k}.
	\end{equation*}
	Therefore, 
	\begin{equation*}
		a_{m,n} \le 2^{n+m}\sum_{k = 1}^{2^n}\sum_{j = 1}^{2^m}\phi^{(n)}_{k,k}\phi^{(m)}_{j,j} = \Big(2^n\sum_{k = 1}^{2^n}\phi^{(n)}_{k,k}\Big)\Big(2^m\sum_{j = 1}^{2^m}\phi^{(m)}_{j,j}\Big).
	\end{equation*}
	By the condition \eqref{lemma_gen_bound}, the double sequence $(a_{m,n})_{m,n \in \bN}$ is uniformly bounded. Thus, there exists $M > 0$ such that $\sup_{m,n} a_{m,n} \le M$. Next, for any $m,n \in \bN$, we have
	\begin{equation*}
		(\phi_{j,k}^{(m,n)})^2 = \Big(\phi^{(m,n+1)}_{j,2k}+\phi^{(m,n+1)}_{j,2k-1}\Big)^2 \le  2\Big((\phi^{(m,n+1)}_{j,2k})^2+(\phi^{(m,n+1)}_{j,2k-1})^2\Big).
	\end{equation*}
	Therefore, one has 
	\begin{equation*}
		a_{m,n} = 2^{n+m}\sum_{k = 1}^{2^n}\sum_{j = 1}^{2^m}(\phi^{(m,n)}_{j,k})^2 \le 2^{n+m+1}\sum_{k = 1}^{2^n}\sum_{j = 1}^{2^m}\Big((\phi^{(m,n+1)}_{j,2k})^2+(\phi^{(m,n+1)}_{j,2k-1})^2\Big) = a_{m,n+1}.
	\end{equation*}
	By an analogous argument, one also gets $a_{m,n} \le a_{m+1,n}$. Therefore, for each fixed $m \in \bN$, there exists a positive non-decreasing sequence $(b_m)_{m \in \bN}$ such that $\lim_{n \ua \infty}a_{m,n} = \lim_{n \ua \infty}a_{n,m} = b_m$. As $b_m \le M$ for all $m \in \bN$, there also exists a positive constant $b$ such that $0 < a_{0,0} \le  b \le M$ and $b_m \ua b$ as $m \ua \infty$. 
	
	Next, we will show the double limit $\lim_{m,n \ua \infty}a_{m,n}$ exists and is equal to $b$. For any $\epsilon > 0$, there exists $m_1(\epsilon)$ such that for all $m \ge m_1(\epsilon)$, we have $0 \le b-b_m \le \epsilon/2$. Moreover, let us select $n_1(\epsilon)$ such that for all $n \ge n_1(\epsilon)$, $0 \le a_{m_1(\epsilon),n} - b_{m_1(\epsilon)} \le \epsilon/2$. Taking $N(\epsilon) := \max[n_1(\epsilon),m_1(\epsilon)]$, one has for $m,n \ge N(\epsilon)$,
	\begin{equation*}
		b-a_{m,n} \le b - a_{m_1(\epsilon),N(\epsilon)} \le (b-b_{m_1(\epsilon)}) + (b_{m_1(\epsilon)}-a_{m_1(\epsilon),N(\epsilon)}) \le \epsilon.
	\end{equation*}
	The first inequality holds as the double sequence $a_{m,n}$ is increasing in each argument. Therefore, the double limit $\lim_{m,n \ua \infty}a_{m,n}$ exists and $\lim_{m,n \ua \infty}a_{m,n} = b > 0$. Thus, by virtue of \cite[Theorem 2.1]{viitasaari2019necessary}, the weighted sum $2^{n}\sum_{k = 1}^{2^n}\big(X\big(\frac{k}{2^n}\big)-X\big(\frac{k-1}{2^n}\big)\big)^2$ converges to some non-constant random variable in $L^p$ for all $p \in (1,\infty)$. This completes the proof.
\end{proof}
Now, we state the convergence results at the critical cases for the trifractional Brownian motion and the $n$-th order fractional Brownian motion.
\begin{theorem}\label{thm_tri_ip}
	For $H \in (0,1)$, $K \in (0,1)$ and $T \in (0,\infty)$, the trifractional Brownian motion $Z_{H,K}$ has the following limiting behavior, 
	\begin{enumerate}
		\item \label{Thm_Tri_ip_a}If $HK < 1/2$, \begin{equation*}
			2^{2HKn}\sum_{k = 1}^{2^n}\left(Z_{H,K}\Big(\frac{kT}{2^n}\Big)-Z_{H,K}\Big(\frac{(k-1)T}{2^n}\Big)\right)^2
		\end{equation*}
		does not converge in probability as $n \ua \infty$.
		\item \label{Thm_Tri_ip_b}If $HK > 1/2$, 
		\begin{equation*}
			2^{n}\sum_{k = 1}^{2^n}\left(Z_{H,K}\Big(\frac{kT}{2^n}\Big)-Z_{H,K}\Big(\frac{(k-1)T}{2^n}\Big)\right)^2
		\end{equation*}
		converges to some non-constant random variable in $L^p$ for all $p \in (1,\infty)$ as $n \ua \infty$.
	\end{enumerate}
\end{theorem}
\begin{theorem}\label{Thm_nfbm_ip}
	For $m \ge 2$, $H \in (m-1,m)$ and $T \in (0,\infty)$, the $m$-th order fractional Brownian motion $B_{H.m}$ has the following limiting behavior, 
	\begin{equation*}
		2^{n} \sum_{k = 1}^{2^n}\bigg(B_{H.m}\Big(\frac{kT}{2^n}\Big) - B_{H.m}\Big(\frac{(k-1)T}{2^n}\Big)\bigg)^2
	\end{equation*}
	converges to some non-constant random variable in $L^p$ for all $p \in (1,\infty)$ as $n \ua \infty$.
\end{theorem}
As observed in Theorem \ref{thm_tri_ip} and Theorem \ref{Thm_nfbm_ip}, the weighted quadratic variation of the trifractional Brownian motion and the $n$-th order fractional Brownian motion behaves essentially different to the Gaussian processes that Gladyshev's theorem apply to. To be more specific, if a Gaussian process fulfills conditions in Gladyshev \cite{gladyshev1961new} or Norvai\v{s}a \cite{norvaivsa2011complement}, its weighted quadratic sum \eqref{eq_weight_sum} converges to some positive constant at the critical case $\alpha = \gamma$ almost surely. However, for the trifractional Brownian motion $Z_{H,K}$, its weighted quadratic variation converges to some non-constant random variable in $L^P$ and in probability for $HK > 1/2$ and does not converge in probability for $HK < 1/2$ at each critical case. Furthermore, as in Theorem \ref{Thm_nfbm_ip}, the weighted sum of the $n$-th order fractional Brownian motion converges to some non-constant random variable at the critical case. These facts further refute the applicability of the existing literature \cite{gladyshev1961new, norvaivsa2011complement} to the trifractional Brownian motion and the $n$-th order fractional Brownian motion.

As previously stated, for many self-similar Gaussian processes, their critical exponents of weighted quadratic variation directly relate to their self-similarity indices. However, a counterexample to this relation is established in part \ref{Thm_Tri_ip_b} of Theorem \ref{Thm_tri_new}, where the critical exponent of the weighted quadratic variation does not yield the self-similarity index of the trifractional Brownian motion. The same is true for the $n$-th order fractional Brownian motion when $n \ge 3$. Nevertheless, an estimator for the self-similarity index $HK$ of the trifractional Brownian motion $Z_{H,K}$ can be constructed for $HK \le \frac{1}{2}$, which is the content of the following corollary.
\begin{corollary}\label{cor_estimate}
	Suppose that $HK \le \frac{1}{2}$, and 
	\begin{equation*}
		\ell := -\lim_{n \ua \infty}\dfrac{\log_2 \sum_{k = 1}^{2^n}\bigg(Z_{H,K}\Big(\frac{kT}{2^n}\Big) - Z_{H,K}\Big(\frac{(k-1)T}{2^n}\Big)\bigg)^2}{2n}
	\end{equation*}
	exists. Then $\ell = HK$ almost surely for every $T > 0$.
\end{corollary}
\begin{proof}
	Take logarithm on both sides of equation \eqref{tri_case_eq_1_1}, we have 
	\begin{equation*}
		\lim\limits_{n \rightarrow \infty} 2n\Big(\frac{\alpha}{2} +\dfrac{\log_2 \sum_{k = 1}^{2^n}\bigg(Z_{H,K}\Big(\frac{kT}{2^n}\Big) - Z_{H,K}\Big(\frac{(k-1)T}{2^n}\Big)\bigg)^2}{2n}\Big) = \begin{cases}
			-\infty &\quad \text{for} \quad\alpha<2HK,\\
			+\infty  &\quad \text{for} \quad \alpha>2HK.
		\end{cases}
	\end{equation*}
	This directly implies the result.
\end{proof}

\section{Proofs}\label{Section_Proofs}

The proofs of the above theorems intensively rely on proper upper bounds of $\phi^{(m,n)}_{j,k}$ for both processes. In the following lemma, we derive upper bounds of $\phi^{(m,n)}_{j,k}$ for the trifractional Brownian motion, which are needed in the proofs of main theorems. Moreover, it suffices to prove all convergence results in the above theorems with $T = 1$. For arbitrary $T > 0$, these results hold due to the homogeneity of covariance functions.
\begin{lemma}\label{lemma_upper_Bound_1}
	For $H \in (0,1)$ and $K \in (0,1)$, we denote  
	\begin{equation*}
		\phi^{(m,n)}_{j,k}:= \bE\left[\Big(Z_{H,K}\Big(\frac{j}{2^m}\Big)-Z_{H,K}\Big(\frac{j-1}{2^m}\Big)\Big)\Big(Z_{H,K}\Big(\frac{k}{2^n}\Big)-Z_{H,K}\Big(\frac{k-1}{2^n}\Big)\Big)\right].
	\end{equation*}
	The following inequalities hold:
	\begin{enumerate}
		\item For $1 < j \le 2^m$ and $1 < k \le 2^n$,
		\begin{equation}\label{eq_upper_Bound_2d}
			\phi^{(m,n)}_{j,k} \le L_12^{-(n+m)HK}(k-1)^{HK-1}(j-1)^{HK-1},
		\end{equation}
		where $L_1 = 2^K(1-K)KH^2 > 0$.
		\item For $1 < j \le 2^m$, $1 < k \le 2^n$ and $H \ge 1/2$, 
		\begin{equation}\label{eq_upper_Bound_2d1}
			\phi^{(m,n)}_{j,k} \le L_2(jk)^{2H-1}\min\Big(\frac{(j-1)^{2HK-4H}}{2^{2Hn+(2HK-2H)m}},\frac{(k-1)^{2HK-4H}}{2^{2Hm+(2HK-2H)n}}\Big),
		\end{equation}
		and for $1 < j \le 2^m$, $1 < k \le 2^n$ and $H < 1/2$,
		\begin{equation}\label{eq_upper_Bound_2d2}
			\phi^{(m,n)}_{j,k} \le L_2\big[(j-1)(k-1)\big]^{2H-1}\min\Big(\frac{(j-1)^{2HK-4H}}{2^{2Hn+(2HK-2H)m}},\frac{(k-1)^{2HK-4H}}{2^{2Hm+(2HK-2H)n}}\Big),
		\end{equation}
		where $L_2 = 4K(1-K)H^2 > 0$.
		\item For $m,n \in \bN$ and $1 < k \le 2^n$, 
		\begin{equation}\label{eq_upper_Bound_1d>}
			\phi^{(m,n)}_{1,k} \le L_3(k-1)^{2HK-1-2H}2^{-2Hm+(2H-2HK)n},
		\end{equation}
		where $L_3 = 2HK(1-K) > 0$. 
		\item For $m,n \in \bN$ and $1 < k \le 2^n$,\begin{equation}\label{eq_upper_Bound_1d2}
			\phi^{(m,n)}_{1,k} \le 2HK(k-1)^{2HK-1}2^{-2HKn}.
		\end{equation}
	\end{enumerate}
\end{lemma}
\begin{proof}[Proof of Lemma \ref{lemma_upper_Bound_1}]
	First, we prove inequality \eqref{eq_upper_Bound_2d}. For $s,t > 0$, we get 
	\begin{equation}\label{eq_partial_C_HK}
		\frac{\partial^2 C_{H,K}}{\partial s \partial t}(s,t) = 4K(1-K)H^2(t^{2H}+s^{2H})^{K-2}(st)^{2H-1} \le 2^K K(1-K)H^2(st)^{HK-1}.
	\end{equation}
	For $1 < j \le 2^m$ and $1 < k \le 2^n$, applying \eqref{eq_partial_C_HK} to the following integral representation leads to
	\begin{equation*}
		\begin{split}
			\phi^{(m,n)}_{j,k} &= C_{H,K}\Big(\frac{k}{2^n},\frac{j}{2^m}\Big)+C_{H,K}\Big(\frac{k-1}{2^n},\frac{j-1}{2^m}\Big)-C_{H,K}\Big(\frac{k}{2^n},\frac{j-1}{2^m}\Big)-C_{H,K}\Big(\frac{k-1}{2^n},\frac{j}{2^m}\Big)\\ &= \int_{(k-1)/2^n}^{k/2^n}\int_{(j-1)/2^m}^{j/2^m}\frac{\partial^2 C_{H,K}}{\partial s \partial t}(s,t)dsdt \le 2^{-(m+n)}\sup_{\substack{k-1 \le 2^ns \le k\\ j-1\le 2^mt \le j}}\frac{\partial^2 C_{H,K}}{\partial s \partial t}(s,t)\\ &\le 2^{-(m+n)}L_1\Big(\frac{(j-1)(k-1)}{2^{m+n}}\Big)^{HK-1} = L_12^{-(m+n)HK}(k-1)^{HK-1}(j-1)^{HK-1}.
		\end{split}
	\end{equation*}
	The last inequality holds as $HK < 1$. Next, we prove \eqref{eq_upper_Bound_2d1} for $H \ge 1/2$. As $K-2 < 0$, one has
	\begin{equation*}
		\frac{\partial^2 C_{H,K}}{\partial s \partial t}(s,t) \le L_2\min\big[s^{2H(K-2)},t^{2H(K-2)}\big](st)^{2H-1}.
	\end{equation*}
	Therefore, 
	\begin{equation*}
		\begin{split}
			\phi^{(m,n)}_{j,k} &\le L_22^{-(m+n)}\sup_{\substack{k-1 \le 2^ns \le k\\ j-1\le 2^mt \le j}}\Big(\min\big[s^{2H(K-2)},t^{2H(K-2)}\big](st)^{2H-1}\Big)\\ &= L_22^{-2H(m+n)}\min\Big[\Big(\frac{j-1}{2^m}\Big)^{2H(K-2)},\Big(\frac{k-1}{2^n}\Big)^{2H(K-2)}\Big]\Big(\frac{jk}{2^{m+n}}\Big)^{2H-1}.
		\end{split}
	\end{equation*}
	Rearranging this inequality gives \eqref{eq_upper_Bound_2d1}, and the proof of \eqref{eq_upper_Bound_2d2} is analogous. Let us now prove \eqref{eq_upper_Bound_1d>}. To this end, for each fixed $m \in \bN$, we take $g(t):= C_{H,K}(1/2^m,t)$. Clearly, $g$ is continuous over $[0,1]$ and differentiable over $(0,1)$. Furthermore, since $K < 1$, we have $[1+(2^{-m}/t)^{2H}]^{K-1}\ge 1 + (K-1)(2^{-m}/t)^{2H}$. Thus,
	\begin{equation*}
		0 < g'(t) = 2HK t^{2HK-1}\left[1 - \Big[1+\Big(\frac{2^{-m}}{t}\Big)^{2H}\Big]^{K-1}\right] \le L_3t^{2HK-1-2H}2^{-2Hm}.
	\end{equation*}
	By the mean value theorem, there exists $\tau \in [(k-1)/2^n,k/2^n]$ such that 
	\begin{equation*}
		\begin{split}
			\phi^{(m,n)}_{1,k} &= C_{H,K}\Big(\frac{1}{2^m},\frac{k}{2^m}\Big)-C_{H,K}\Big(\frac{1}{2^m},\frac{k-1}{2^m}\Big) = g'(\tau)2^{-n}\\ &\le L_3(k-1)^{2HK-1-2H}2^{-2Hm+(2H-2HK)n},
		\end{split}
	\end{equation*}
	where the last inequality holds as $2HK-1-2H < 0$. This completes the proof of \eqref{eq_upper_Bound_1d>}. Finally, as $g'(t) \le 2HKt^{2HK-1}$, it leads to 
	\begin{equation*}
		\phi^{(m,n)}_{1,k} \le 2HK\Big(\frac{k-1}{2^n}\Big)^{(2HK-1)n}2^{-n} = 2HK(k-1)^{2HK-1}2^{-2HKn}.
	\end{equation*}
	Thus, we complete the proof of \eqref{eq_upper_Bound_1d2}.
\end{proof}
\begin{proof}[Proof of Theorem \ref{Thm_tri_new}]
	For $n \in \bN$ and $\alpha \in \bR$, we let 
	\begin{equation}\label{eq_def_S}
		S^\alpha_n:= 2^{\alpha n} \sum_{k = 1}^{2^n}\bigg(Z_{H,K}\Big(\frac{k}{2^n}\Big) - Z_{H,K}\Big(\frac{k-1}{2^n}\Big)\bigg)^2.
	\end{equation}
	As the trifractional Brownian motion is a centered Gaussian process, it then follows from Isserlis' theorem
	\begin{equation}\label{eq_exp_var_rep}
		\bE(S_n^\alpha)= 2^{\alpha n}\sum_{k = 1}^{2^n}\phi^{(n)}_{k,k} \quad \text{and} \quad \var{(S_n^\alpha)}= 2^{2\alpha n +1}\sum_{j,k=1}^{2^n}(\phi^{(n)}_{j,k})^2.
	\end{equation}
	In the sequel, we start by analyzing the limits of $\bE(S_n^\alpha)$ and $\var{(S_n^\alpha)}$ as $n \ua \infty$. Subsequently, we discuss case by case to prove \eqref{tri_case_eq_1_1} and \eqref{tri_case_eq_2_1}. Finally, we prove \eqref{tri_case_eq_1_2} and \eqref{tri_case_eq_2_2}. First, applying \eqref{eq_upper_Bound_2d} gives 
	\begin{equation}\label{eq_tri_as_1}
		\bE(S_n^\alpha)= 2^{\alpha n}\Big(\phi^{(n)}_{1,1}+\sum_{k = 2}^{2^n}\phi^{(n)}_{k,k}\Big) \le 2^{(\alpha-2HK) n}\Big((2-2^K)+L_1\sum_{k = 1}^{2^n-1}k^{2HK-2}\Big).
	\end{equation}
	Seeing that for $2HK-2 < -1$, one has $\sum_{k = 1}^{\infty}k^{2HK-2} < \infty$. This implies that if $HK < 1/2$, we have $\bE(S_n^\alpha)\da0$ for $\alpha < 2HK$ as $n \ua \infty$. For $HK > 1/2$, one has
	\begin{equation}\label{eq_tri_as_2}
		\sum_{k = 1}^{2^n-1}k^{2HK-2} \le \int_{0}^{2^n}t^{2HK-2}dt = (2HK-1)^{-1}2^{(2HK-1)n}.
	\end{equation}
	Inequalities \eqref{eq_tri_as_1} and \eqref{eq_tri_as_2} then imply that for $HK > 1/2$, $\bE(S_n^\alpha)\da0$ for $\alpha < 1$ as $n \ua \infty$. Last, for $HK = 1/2$, we have
	\begin{equation*}
		\bE(S_n^\alpha) \le 2^{(\alpha-1)n}\Big((2-2^K)+L_1\sum_{k = 1}^{2^n}\frac{1}{k}\Big) = 2^{(\alpha-1)n}\big((2-2^K)+L_1h_{2^n}\big),
	\end{equation*}
	where $h_{2^n}$ denotes the $2^n$-th \textit{Harmonic number}, and $h_{2^n} \sim n\log2$ as $n \ua \infty$. Thus, we have $\bE(S_n^\alpha) \da 0$ as $n \ua \infty$ for $\alpha < 1$ and $HK = 1/2$. Next, we discuss the upper bounds for $\var{(S^\alpha_n)}$ case by case. By the Cauchy-Schwarz inequality, one has $(\phi^{(n)}_{1,k})^2 \le \phi^{(n)}_{1,1}\phi^{(n)}_{k,k}$. Then,
	\begin{equation*}
		\var{(S_n^\alpha)} \le 2^{2\alpha n +1}\Big(\sum_{j,k=2}^{2^n}(\phi^{(n)}_{j,k})^2 + 2\sum_{k = 1}^{2^n}(\phi^{(n)}_{1,k})^2\Big) \le 2^{2\alpha n +1}\Big(\sum_{j,k=2}^{2^n}(\phi^{(n)}_{j,k})^2 + 2\sum_{k = 1}^{2^n}\phi^{(n)}_{1,1}\phi^{(n)}_{k,k}\Big),
	\end{equation*}
	From \eqref{eq_tri_as_1}, one get
	\begin{equation}\label{eq2}
		\begin{split}
			2^{2\alpha n}\sum_{k = 1}^{2^n}\phi^{(n)}_{1,1}\phi^{(n)}_{k,k} &\le 2^{(2\alpha-4HK)n}(2-2^K)\Big((2-2^K)+L_1\sum_{k = 1}^{2^n-1}k^{2HK-2}\Big)\\ &=  \begin{cases}
				\mathcal{O}(2^{2(\alpha - 2HK)n})&\quad \text{for}\quad HK \in (0,1/2),\\ \mathcal{O}(2^{2(\alpha -1)n}n) &\quad \text{for}\quad HK = 1/2,\\ \mathcal{O}(2^{(2\alpha -2HK-1)n}) &\quad \text{for}\quad HK \in (1/2,1),
		\end{cases} \end{split}
	\end{equation}
	as $n \ua \infty$. Moreover, following \eqref{eq_upper_Bound_2d}, we have 
	\begin{equation}\label{eq3}
		\begin{split}
			2^{2\alpha n}\sum_{j,k=2}^{2^n}(\phi^{(n)}_{j,k})^2 &\le L_1^2\sum_{j,k=1}^{2^n}2^{(2\alpha-4HK)n}(jk)^{4HK-2} = \Big(L_12^{(\alpha-2HK)n}\sum_{k = 1}^{2^n}k^{2HK-1}\Big)^2 \\&= \begin{cases}
				\mathcal{O}(2^{2(\alpha - 2HK)n})&\quad \text{for}\quad HK \in (0,1/2),\\ \mathcal{O}(2^{2(\alpha -1)n}n^2) &\quad \text{for}\quad HK = 1/2,\\ \mathcal{O}(2^{2(\alpha -1)n}) &\quad \text{for}\quad HK \in (1/2,1),
			\end{cases} 
		\end{split}
	\end{equation}
	as $n \ua \infty$. Subsequently, the results \eqref{tri_case_eq_1_1} and \eqref{tri_case_eq_2_1} follow directly from the Borel-Cantelli lemma. By virtue of the Borel-Cantelli lemma, the fast $L^2$-convergence, i.e., $\sum_{n = 0}^{\infty}\var{(S_n^\alpha)} < \infty$, guarantees the almost sure convergence. For instance, inequalities \eqref{eq_tri_as_1}, \eqref{eq2} and \eqref{eq3} imply that $\bE(S^\alpha_n) \da 0$ and $\sum_{n = 0}^{\infty}\var{(S^\alpha_n)}<\infty$ for $\alpha < 2HK$ and $HK < 1/2$. Therefore, $S^\alpha_n$ converges to zero almost surely. The proofs for $HK = 1/2$ and $HK \ge 1/2$ also follow from \eqref{eq_tri_as_1}, \eqref{eq2} and \eqref{eq3} analogously. Thus, we complete the proof of \eqref{tri_case_eq_1_1} and \eqref{tri_case_eq_2_1}.
	
	Now, let us proceed to prove \eqref{tri_case_eq_1_2}. There exists a (dependent) sequence of standard normally distributed random variables $\{Y_n, n \ge 1\}$, such that 
	\begin{equation}\label{tri_dist_eq}
		Z_{H,K}\Big(\frac{1}{2^n}\Big) = \sqrt{2-2^K}2^{-HKn}Y_n.
	\end{equation}
	Therefore, for $HK < \frac{1}{2}$, $\alpha >2HK$ and any given $M > 0$, we get
	\begin{align*}
		\bP\big(S_n^\alpha \le M\big) &\le \bP\big(2^{\alpha n} Z^2_{H,K}\Big(\frac{1}{2^n}\Big) \le M\big)= \bP\Big(Y^2_n \le \frac{2^{(2HK-\alpha)n} M}{(2-2^K)}\Big) = \frac{\gamma\Big(1/2,\frac{2^{(2HK-\alpha)n} M}{2(2-2^K)}\Big)}{\Gamma(1/2)},
	\end{align*}
	where $\gamma(k,x) = \int_{0}^{x}t^{k-1}e^{-t}dt$ is the \emph{lower incomplete gamma function}. By change of variables and Taylor expansion, we have 
	\begin{equation*}
		\dfrac{\gamma\Big(1/2,\dfrac{2^{(2HK-\alpha)n} M}{2(2-2^K)}\Big)}{\Gamma(1/2)} = \frac{2}{\sqrt{\pi}}\int_0^{\sqrt{\frac{2^{(2HK-\alpha)n} M}{2(2-2^K)}}}e^{-t^2}dt = \frac{2}{\sqrt{\pi}}\sqrt{\frac{2^{(2HK-\alpha)n} M}{2(2-2^K)}} + \mathcal{O}(2^{(3HK-\frac{3}{2}\alpha)n}),
	\end{equation*}as $n \ua \infty$. As $2HK-\alpha < 0$, the Borel-Cantelli lemma implies that $S_n^\alpha$ converges to infinity almost surely. This completes the proof of \eqref{tri_case_eq_2_1}.
	
	Before proving equation \eqref{tri_case_eq_2_2}, let us recall the definition of the Lei-Nualart process $X_K = \{X_K(t), t\in [0,\infty)\}$. For $K \in (0,1)$, the Lei-Nualart process is defined as 
	\begin{equation}\label{X_process_eq}
		X_K(t) = \int_{0}^{\infty} (1 - e^{-s t})s^{-\frac{1+K}{2}}dB_s, 
	\end{equation}
	where $\{B_t, t \in [0,\infty)\}$ is a standard Brownian motion. Lei and Nualart \cite{lei2009decomposition} show that the process $X_K$ has trajectories which are infinitely differentiable on $(0,\infty)$ and absolutely continuous on $[0,\infty)$. As a result of the following construction in \cite{lei2009decomposition}
	\begin{equation*}
		Z_{H,K}(t) = \sqrt{\frac{K}{\Gamma(1-K)}}X_{K}(t^{2H}),
	\end{equation*}
	the trifractional Brownian motion $Z_{H,K}(t)$ also admits infinitely differentiable trajectories on $(0,\infty)$. Hence, for any $n \in \bN$ and $k \ge 2$, we have
	\begin{equation*}
		Z_{H,K}\Big(\frac{k}{2^n}\Big) = Z_{H,K}\Big(\frac{k-1}{2^n}\Big) + Z'_{H,K}\Big(\frac{k-1}{2^n}\Big)2^{-n} + \frac{1}{2}Z''_{H,K}\Big(\frac{u_{k,n}}{2^n}\Big)2^{-2n},
	\end{equation*}
	where $u_{k,n} \in (k-1,k)$. Hence, 
	\begin{equation*}
		\bigg(Z_{H,K}\Big(\frac{k}{2^n}\Big)-Z_{H,K}\Big(\frac{k-1}{2^n}\Big)\bigg)^2 = 2^{-2n}\Big(Z'_{H,K}\Big(\frac{k-1}{2^n}\Big)\Big)^2 + \mathcal{O}(2^{-3n}).
	\end{equation*}
	As $\big(Z'_{H,K}(t)\big)^2$ is still continuously differentiable on $(0,\infty)$ by composition, we have for every $\omega \in \Omega$ and any $\varepsilon > 0$, 
	\begin{align*}
		&\lim_{n\rightarrow \infty}\sum_{k = \left \lfloor{2^n\epsilon}\right \rfloor }^{2^n}2^n\bigg(Z_{H,K}\Big(\frac{k}{2^n},\omega\Big) - Z_{H,K}\Big(\frac{k-1}{2^n},\omega\Big)\bigg)^2\\&= \lim_{n\rightarrow \infty}\sum_{k = \left \lfloor{2^n\epsilon}\right \rfloor }^{2^n}2^{-n}\Big(Z'_{H,K}\Big(\frac{k-1}{2^n},\omega\Big)\Big)^2 + \mathcal{O}(2^{-n}) = \int_{\varepsilon}^{1}\Big(Z'_{H,K}(t,\omega)\Big)^2dt.
	\end{align*}
	As $Z_{H,K}(t)$ is non-constant almost surely,  $\int_{\varepsilon}^{1}\big(Z'_{H,K}(t)\big)^2dt$ is positive almost surely. As a result, for $\alpha > 1$, we have 
	\begin{align*}
		\lim_{n\rightarrow \infty}S^\alpha_n &\ge \lim_{n\rightarrow \infty}2^{(\alpha-1)n}\lim_{n\rightarrow \infty}\sum_{k = \left \lfloor{2^n\epsilon}\right \rfloor }^{2^ n}2^{ n}\bigg(Z_{H,K}\Big(\frac{k}{2^n}\Big) - Z_{H,K}\Big(\frac{k-1}{2^n}\Big)\bigg)^2 \\&= \lim_{n\rightarrow \infty}2^{(\alpha-1)n}\int_{\varepsilon}^{1}\Big(Z'_{H,K}(t)\Big)^2dt = \infty, \quad a.s.
	\end{align*}
	This completes the proof of the result \eqref{tri_case_eq_2_2}. Hence, we complete the proof of Theorem \ref{Thm_tri_new}. 
\end{proof}
\begin{proof}[Proof of Theorem \ref{nfbm_thm}]
	Throughout the following proof, we adopt the notations in \eqref{eq_phi_def} and \eqref{eq_def_S} with the $n$-th order fractional Brownian motion. Applying Taylor series expansion to \eqref{n_fbm_eq_1} with $z = 1/t$ yields
	\begin{align*}
		\psi_{B_{H,m}}(t,t-1) &= (-1)^mC_H^m\Bigg(\sum_{j = 0}^{m-1}(-1)^j{2H \choose j}\bigg(t^j(t-1)^{2H-j} + (t-1)^jt^{2H-j} -t^{2H} - (t-1)^{2H}\bigg) - 1\Bigg)\\&=(-1)^mC_H^m\Bigg(\sum_{j = 0}^{m-1}(-1)^j{2H \choose j}z^{-2H}\Big((1-z)^{2H-j} + (1-z)^{j} - 1 - (1-z)^{2H}\Big) - 1\Bigg)\\&=(-1)^mC_H^m\Bigg(\sum_{j = 0}^{m-1}(-1)^j{2H \choose j}z^{-2H}\Big(j(j-2H)z^2 + \mathcal{O}(z^3)\Big) - 1\Bigg)\\&=(-1)^mC_{H}^m\Bigg(2H(2H-1)\sum_{j=0}^{m-2}(-1)^{j}{2H-2 \choose j} t^{2H-2} -1\Bigg) + \mathcal{O}(t^{2H-3})\\&= C_{H-1}^{m-1}{2H-3 \choose m-2}t^{2H-2} + (-1)^{m+1}C_H^m + \mathcal{O}(t^{2H-3}).
	\end{align*}
	As $m \ge 2$ and $H \in (m-1,m)$, we have $2H - 2 > 0$, and this gives
	\begin{equation*}
		\psi_{B_{H,m}}(t,t-1) \sim C_{H-1}^{m-1}{2H-3 \choose m-2}t^{2H-2} \quad \text{as} \quad t \ua \infty.
	\end{equation*}
	As the $m$-th order fractional Brownian motion is $H$-self-similar, then 
	\begin{equation}\label{fbm_eq_1}
		\sum_{k = 1}^{2^n}\phi^{(n)}_{k.k} = 2^{-2Hn}\sum_{k = 1}^{2^n}\psi_{B_{H,m}}(k-1,k) = \mathcal{O}(2^{-n}) \quad \text{as} \quad n \ua \infty.
	\end{equation}
	Therefore, $\bE(S_n^\alpha) = \mathcal{O}(2^{(\alpha-1)n})$ as $n \ua \infty$. Now, let us recall that the Cauchy-Schwarz inequality gives $(\phi^{(n)}_{j,k})^2 \le \phi^{(n)}_{j,j}\phi^{(n)}_{k,k}$, which leads to
	\begin{equation*}
		\var(S_n^\alpha) = 2^{2\alpha n+1}\sum_{j,k = 1}^{2^n}(\phi^{(n)}_{j,k})^2 \le 2^{2\alpha n+1}\left(\sum_{k = 1}^{2^n}\phi^{(n)}_{k,k}\right)^2 = \mathcal{O}(2^{2(\alpha-1)n}).
	\end{equation*}
	This then gives the fast $L^2$-convergence for the case $\alpha < 1$, which guarantees the almost sure convergence. Hence, the result in \eqref{nthm_thm_eq_1} is then verified. 
	
	The proof of \eqref{nthm_thm_eq_2} follows analogously as the arguments in the proof of \eqref{tri_case_eq_2_2} in Theorem \ref{Thm_tri_new}. Following \cite[Remark 2.3]{sottinen2019transfer}, $B_{H,m}$ is $(m-1)$ times differentiable. Thus, for $m \ge 3$, the arguments in the proof Theorem \ref{Thm_tri_new} can be applied to the $m$-th order fractional Brownian motion. Hence, we complete the proof of Theorem \ref{nfbm_thm}.
\end{proof}
\begin{proof}[Proof of Theorem \ref{thm_tri_ip}]
	First, we prove \ref{Thm_Tri_ip_a} for the case $2H < 1$, the proof for the case $2H \ge 1$ is analogous. By virtue of Viitasaari \cite[Theorem 2.1]{viitasaari2019necessary}, it suffices to show that the double limit
	\begin{equation}\label{eq_double}
		\lim_{n,m \ua \infty}2^{2HK(n+m)}\sum_{k = 1}^{2^n}\sum_{j = 1}^{2^m}(\phi^{(m,n)}_{j,k})^2
	\end{equation}
	fails to exist. To this end, for $n,m \in \bN$, we denote $a_{m,n}:= 2^{2HK(n+m)}\sum_{k = 1}^{2^n}\sum_{j = 1}^{2^m}(\phi^{(m,n)}_{j,k})^2$. Next, we show the limit $\lim_{n \ua \infty}a_{n,n}$ exists and is strictly positive. By the Cauchy-Schwarz inequality, we have $(\phi_{j,k}^{(m,n)})^2 \le \phi^{(m)}_{j,j}\phi^{(n)}_{k,k}$, then
	\begin{equation*}
		a_{n,m} \le 2^{2HK(n+m)}\sum_{k = 1}^{2^n}\sum_{j = 1}^{2^m}\phi^{(n)}_{k,k}\phi^{(m)}_{j,j} = \Big(2^{2HKn}\sum_{k = 1}^{2^{n}}\phi^{(n)}_{k,k}\Big)\Big(2^{2HKm}\sum_{j = 1}^{2^{m}}\phi^{(m)}_{j,j}\Big).
	\end{equation*}
	From \eqref{eq_tri_as_1}, if $HK < 1/2$, the sequence $(\sum_{j = 1}^{2^{m}}\phi^{(m)}_{j,j})_{m \in \bN}$ is uniformly bounded. Hence, there exists $M > 0$, such that $\sup_{m,n}a_{m,n} \le M$. As the trifractional Brownian motion is $HK$-self-similar, then
	\begin{equation*}
		a_{n,n} = \sum_{k = 1}^{2^n}\sum_{j = 1}^{2^n}(2^{2HKn}\phi^{(n)}_{j,k})^2 = \sum_{k = 1}^{2^n}\sum_{j = 1}^{2^n}\bE\Big[\Big(Z_{H,k}(j)-Z_{H,k}(j-1)\Big)\Big(Z_{H,k}(k)-Z_{H,k}(k-1)\Big)\Big]^2.
	\end{equation*}
	Therefore, $(a_{n,n})_{n \in \bN}$ forms a non-decreasing uniformly bounded sequence. Hence, the limit $\lim_{n \ua \infty}a_{n,n}$ exists and $\lim_{n \ua \infty}a_{n,n} \ge a_{0,0} = \bE([Z_{H,K}(1)]^2) = (2-2^K) > 0$. 
	
	Now, we show that for each fixed $n \in \bN$, $\lim_{m \ua \infty}a_{m,n} = \lim_{m \ua \infty}a_{n,m} = 0$. For $k \ge 2$, inequality \eqref{eq_upper_Bound_2d2} yields
	\begin{equation*}
		2^{2HK(m+n)}\sum_{j = 1}^{2^m}(\phi_{j,k}^{(m,n)})^2 \le(k-1)^{4HK-4H-2}2^{(4H-2HK)n}2^{(2HK-4H)m}\Big( L_2^2\sum_{j = 1}^{2^m}j^{4H-2} + L_3^2\Big).
	\end{equation*}
	If $H < 1/4$, the sequence $(j^{4H-2})_{j \in \bN} \in \ell_1$, and $2^{2HK(m+n)}\sum_{j = 1}^{2^m}(\phi_{j,k}^{(m,n)})^2$ converges to zero as $m \ua \infty$. If $H > 1/4$, one has
	\begin{equation*}
		2^{(2HK-4H)m}\sum_{j = 1}^{2^m}j^{4H-2} \le 2^{(2HK-4H)m}\int_{0}^{2^m}t^{4H-2}dt = (4H-1)^{-1}2^{(2HK-1)m}.
	\end{equation*}
	Therefore, $2^{2HK(m+n)}\sum_{j = 1}^{2^m}(\phi_{j,k}^{(m,n)})^2 = \mathcal{O}(2^{(2HK-1)m})$ converges to zero as $m \ua \infty$. Last, if $H = 1/4$, then $\sum_{j = 1}^{2^m}j^{-1} = h_{2^m} \sim m\log 2$, where $h_{2^m}$ is the $2^m$-th Harmonic number. It then implies that $2^{2HK(m+n)}\sum_{j = 1}^{2^m}(\phi_{j,k}^{(m,n)})^2$ converges to zero. For $k = 1$, it follows from \eqref{eq_upper_Bound_1d2} that
	\begin{equation}
		\begin{split}
			2^{2HK(m+n)}\sum_{j = 2}^{2^m}(\phi_{j,1}^{m,n})^2 &\le 2^{2HK(n-m)+2}H^2K^2\sum_{j = 2}^{2^m}(j-1)^{4HK-2}\\ &= \begin{cases} \mathcal{O}(2^{-2HKm}) \quad &\text{for} \quad HK < 1/4,\\
				\mathcal{O}(2^{-2HKm}m) \quad &\text{for} \quad HK = 1/4,\\
				\mathcal{O}(2^{(2HK-1)m}) \quad &\text{for} \quad HK > 1/4,
			\end{cases}
		\end{split}
	\end{equation}
	as $m \ua \infty$. Finally, for each $n \in \bN$, we have 
	\begin{equation}
		0 < 2^{2HK(m+n)}(\phi^{(m,n)}_{1,1})^2 \le 2^{2HK(n-m)}.
	\end{equation}
	The above inequalities in together imply that for each $n \in \bN$ and $1 \le k \le 2^n$, we have 
	\begin{equation*}
		\lim_{m \ua \infty}2^{HK(m+n)}\sum_{j = 1}^{2^m}(\phi_{j,k}^{(m,n)})^2 = 0.
	\end{equation*}
	Hence, we have $\lim_{m \ua \infty}a_{m,n} = \lim_{m \ua \infty}a_{n,m} = 0$. However, as previously proved, $\lim_{n \ua \infty}a_{n,n} > 0$. Therefore, the double limit \eqref{eq_double} does not exists, and we then complete the proof of assertion \ref{Thm_Tri_ip_a}. 
	
	We now prove assertion \ref{Thm_Tri_ip_b}. Following \eqref{eq_tri_as_1} and \eqref{eq_tri_as_2}, we have 
	\begin{equation*}
		\sup_{n \in \bN} 2^n \sum_{k = 1}^{2^n}\phi^{(n)}_{k,k} \le (2-2^K)+\frac{L_1}{2HK-1}.
	\end{equation*}
	Then Proposition \ref{Lemma_Lp} implies the assertion \ref{Thm_Tri_ip_b}. 
\end{proof}
\begin{proof}[Proof of Theorem \ref{Thm_nfbm_ip}]
	This theorem follows as a direct corollary of Proposition \ref{Lemma_Lp} and \eqref{fbm_eq_1}.
\end{proof}
\bibliographystyle{plain}
\bibliography{CTbook}
\end{document}